\documentclass{amsart}

\usepackage{amsmath,amsfonts,amsthm,amssymb,cite}
\usepackage{mathrsfs}
\usepackage{mathtools}

\theoremstyle{definition}
\newtheorem{thm}{Theorem}[section]
\newtheorem{prop}[thm]{Proposition}

\newtheorem{defn}{Definition}[section]

\newcommand{\gl}{\ensuremath{\mathscr{L}}} 
\newcommand{\gr}{\ensuremath{\mathscr{R}}}
\newcommand{\gh}{\ensuremath{\mathscr{H}}}
\newcommand{\gd}{\ensuremath{\mathrel{\mathscr{D}}}}
\newcommand{\tls}{\ensuremath{T\mathbb{L}(S)}}
\newcommand{\trs}{\ensuremath{T\mathbb{R}(S)}}
\newcommand{\los}{\ensuremath{\mathbb{L}(S)}}
\newcommand{\ros}{\ensuremath{\mathbb{R}(S)}}
\newcommand{\nls}{\ensuremath{N^*\mathbb{L}(S)}}
\newcommand{\nrs}{\ensuremath{N^*\mathbb{R}(S)}}

\begin{document}

\author{P.A. Azeef Muhammed}
\address{Centre for Research in Mathematics and Data Science, Western Sydney University, Penrith, NSW 2751, Australia.}

\author{C.S. Preenu}
\address{Department of Mathematics, University College, Thiruvananthapuram, Kerala 695034, India.}

\title{Cross-connections in Clifford  Semigroups}

\maketitle

\begin{abstract}
An inverse Clifford semigroup (often referred to as just a Clifford semigroup) is a semilattice of groups. It is an inverse semigroup and in fact, one of the earliest studied classes of semigroups~\cite{clifinv}. In this short note, we discuss various structural aspects of a Clifford semigroup from a cross-connection perspective. In particular, given a Clifford semigroup $S$, we show that the semigroup $\tls$ of normal cones is isomorphic to the original semigroup $S$, even when $S$ is not a monoid. Hence, we see that cross-connection description degenerates in Clifford semigroups. Further, we specialise the discussion to provide the description of the cross-connection structure in an arbitrary semilattice, also.  
\end{abstract}

\section{Introduction}
Grillet introduced  cross-connections as a pair of functions to describe the interrelationship between the posets of principal left and right  ideals of a regular semigroup. This construction involved building two intermediary semigroups and further identifying a fundamental image of the semigroup as a subdirect product, using the cross-connection functions. But isomorphic posets give rise to isomorphic cross-connections. 
So, Nambooripad replaced posets with certain small categories to overcome this limitation. Hence, using the categorical theory of cross-connections, Nambooripad constructed arbitrary semigroups from their ideal structure.  

Starting from a regular semigroup $S$, Nambooripad identified two small categories: $\los$ and $\ros$ which abstract the principal left and right ideal structures of the semigroup, respectively. He showed that these categories are interconnected using a pair of functors. It can be seen that the object maps of these functors coincides with Grillet's cross-connection functions and so Nambooripad called his functors also cross-connections. Further, Nambooripad showed that this correspondence can be extended to an explicit category equivalence between the category of regular semigroups and the category of cross-connections.  Hence the ideal structure of a regular semigroup can be completely captured using these `cross-connected' categories.

Being a rather technical construction, it is instructive to work out the simplifications that arise in various special classes of semigroups. There has been several works in this direction \cite{tx,tlx,rajan2022cross,css,locinverse,var}
and in this short article, we propose to outline how the construction simplifies in a couple of very natural classes of regular semigroups: namely Clifford semigroups and semilattices. As the reader may see, this exercise also provides some useful illustrations to several cross-connection related subtleties.

In fact, Clifford semigroups are one of the first classes of inverse semigroups whose structure was studied. It was originally defined \cite{clifinv} as a union of groups in which idempotents commute. It may be noted that some authors refer to general union of groups also as Clifford semigroups but we shall follow \cite{howie1995fundamentals} and refer to a semilattice of groups as a Clifford semigroup. The following characterizations of Clifford semigroups will be useful in the sequel.
\begin{thm}\cite[Theorem 4.2.1]{howie1995fundamentals}\cite[Theorem 1.3.11]{higgins1992techniques}
Let $S$ be a semigroup. Then the following statements are equivalent.
\begin{enumerate}
    \item $S$ is a Clifford semigroup;
    \item $S$ is a semilattice of groups;
    \item $S$ is a strong semilattice of groups;
    \item $S$ is regular and the idempotents of $S$ are central;
    \item Every $\gd$ class of $S$ has a unique idempotent. 
\end{enumerate}
\end{thm}

In Nambooripad's cross-connection description, starting from a regular semigroup $S$, two small categories of principal left and right ideals (denoted by $\los$ and $\ros$, respectively in the sequel) are defined. Then their inter-relationship is abstracted as a pair of functors called cross-connections. Conversely, given an abstractly defined pair of cross-connected categories (with some special properties), one can construct a regular semigroup. This correspondence between regular semigroups and cross-connections is proved to be a category equivalence. 

The construction of the regular semigroup from the category happens in several layers and the crucial object here is the intermediary regular semigroup $\tls$ of normal cones from the given category $\los$. We shall see that when the semigroup $S$ is Clifford, then the semigroup $\tls$ is isomorphic to $S$. It is known that this is not true in general  \cite{locinverse} even for an inverse semigroup $S$. This in turn, highly degenerates the cross-connection structure in Clifford semigroups. This is discussed in the next section. In the last section, we specialise our discussion to an arbitrary semilattice and describe the cross-connections therein.

As mentioned above, since the article can be seen as a part of a continuing project of studying the various classes of regular semigroups within the cross-connection framework, we refer the reader to \cite{locinverse,rajan2022cross,var} for the preliminary notions and formal definitions in Nambooripad's cross-connection theory. We also refer the reader to \cite{nambooripad1994theory} for the original treatise on cross-connections.

\section{Normal categories and cross-connections in Clifford semigroups}

Recall from \cite{nambooripad1994theory} that given a regular semigroup $S$, the normal category $\los$ of principal left ideals are defined as follows. The object set 
$$v\los:=~~\{ Se : e \in E(S) \}$$  and the morphisms in $\los$ are partial right translations. In fact, the set of all morphisms between two objects $Se$ and $Sf$ may be characterised as the set $ \{ \rho(e,u,f) : u\in eSf \}$, where the map $\rho(e,u,f)$
sends $x\in Se$ to $xu\in Sf$. 

First, we proceed to discuss some special properties of the category $\los$ when $S$ is a Clifford semigroup. This will lead us to the characterisation of the semigroup  $\tls$ of all normal cones in $S$.

\begin{prop}
Let $S$ be a Clifford semigroup. Two objects in $\los$ are isomorphic  if and only if they are identical. 
\end{prop}
\begin{proof}
  Clearly, identical objects are always isomorphic. Conversely, suppose $Se$ and $Sf$ are two isomorphic objects in $\los$. Then by \cite[Proposition III.13(c)]{nambooripad1994theory}, we have  $e\gd f$ . Recall that in a Clifford semigroup, the Green's relations $\gl$, $\gr$ and $\gd$ are identical. Therefore $e\gl f$ and hence $Se=Sf$.   
\end{proof}

Given the normal category $\los$ of principal left ideals of a regular semigroup $S$, it is known that two morphisms are equal, i.e. $\rho(e,u,f)=\rho(g,v,h)$ if and only if $e\mathrel{\mathscr{L}}g$, $f\mathrel{\mathscr{L}}h$, and $v=gu$. Also, given two morphisms $\rho(e,u,f)$ and $\rho(g,v,h)$, they are composable if and only $Sf=Sg$ so that $\rho(e,u,f)\rho(g,v,h) = \rho(e,uv,h)$. Now, we see that the equality of morphisms simplify when $S$ is a Clifford semigroup.

\begin{prop}\label{proeq}
Let $S$ be a Clifford semigroup, then  $\rho(e,u,f)=\rho(g,v,h)$  in the category $\los$ if and only if $e=g$, $u=v$ and $f=h$.
\end{prop}
\begin{proof}
Suppose that two morphisms $\rho(e,u,f)=\rho(g,v,h)$ are equal. Recall that by \cite[Lemma II.12]{nambooripad1994theory}, this implies that $Se=Sg$, $Sf=Sh$ and $v=gu$ . That is, the elements $e$ and $g$ are two $\gl$ related idempotents. But since $\gl$ and $\gh$ are identical in a Clifford semigroup $S$ and since a $\gh$-class can contain at most one idempotent, we have $e=g$. Similarly we get $f=h$. So, the sets $eSf=gSh$ are equal and also $v=gu=eu=u$. 
\end{proof}

Now, we proceed to characterise the building blocks of the cross-connection construction, namely the normal cones in the category $\los$. An `order-respecting' collection of morphisms in a normal category is defined as a normal cone. 

\begin{defn}
\label{def:normalcone}
Let $\los$ be the normal category of principal left ideals in a regular semigroup $S$ and $Sd\in v\los$. A \emph{ normal cone with apex $Sd$} is a function $\gamma\colon v\los\to \los$ such that:
	\begin{enumerate}
	    \item for each $Se\in v\los$, one has $\gamma(Se)\in\los(Se,Sd)$;
	    \item $\iota(Sf,Sg) \gamma(Sg)=\gamma(Sf)$ whenever $Sf\subseteq Sg$;
	    \item $\gamma(Sm)$ is an isomorphism for some $Sm\in v\los$.
	\end{enumerate}
\end{defn}

Now, for each $a\in S$, we can define a function $\rho^a\colon v\mathbb{L}(S)\to \mathbb{L}(S)$ as follows:
\begin{equation}\label{eqnprinc}
\rho^a(Se):= \rho(e,ea,f) \text{ where } f\mathrel{\mathscr{L}}a.
\end{equation}
It is easy to verify that the map $\rho^a$ is a well-defined normal cone with apex $Sf$ in the sense of Definition~\ref{def:normalcone}, see~\cite[Lemma III.15]{nambooripad1994theory}.

In the sequel, the normal cone $\rho^a$ is called the \emph{principal cone} determined by the element $a$. In particular, observe that, for an idempotent $e\in E(S)$, we have a principal cone  $\rho^e$ such that $\rho^e(Se)= \rho(e,e,e)=1_{Se}$. This leads us to the most crucial proposition of this article.

\begin{prop}\label{pric:cone}
In a Clifford semigroup $S$, every normal cone in $\los$ is a principal cone.  
\end{prop}
\begin{proof}
	Suppose $\gamma$ is a normal cone in $\los$  with vertex $Se$ so that $\gamma(Se)=\rho(e,u,e)$ for some $u\in S$. 
	Then for any $Sf\in v\los$, we shall show that $\gamma(Sf)=\rho(f,fu,e)=\rho^u$ for some $u\in S$. To this end, first observe that since idempotents commute in $S$, we have $Sef=Sfe\subseteq Se$. So, by (2) of Definition \ref{def:normalcone}, we have 			$\gamma(Sef)=\rho(ef,ef,e)\gamma(Se)$. Then,
	\begin{align*}
		\gamma(Sef)&=\rho(ef,ef,e)\gamma(Se)& \text{ since }Sef\subseteq Se\\
		&=\rho(ef,ef,e)\rho(e,u,e)\\
		&=\rho(ef,efu,e)\\
		&=\rho(ef,fu,e)&\text{ since }u\in eSe\text{ and }ef=fe.
	\end{align*}
Now for any $Sf\in v\los$, let $\gamma(Sf)=\rho(f,v,e)$ for some $v\in fSe$. Then since $Sef\subseteq Sf$ also, we have
\begin{align*}
	\gamma(Sef)&=\rho(ef,ef,f)\gamma(Sf)\\
	&=\rho(ef,ef,f)\rho(f,v,e)\\
	&=\rho(ef,efv,e)\\
	&=\rho(ef,ev,e)&\text{since }fv=v\\
	&=\rho(ef,ve,e)&\text{since idempotents are central in }S\\
	&=\rho(ef,v,e)&\text{ since }v\in fSe.
\end{align*}
Now, from the discussion above, we see that $\rho(ef,fu,e)=\rho(ef,v,e)$. Then, using Proposition \ref{proeq}, this implies that $v=fu$. Hence for all $Sf\in v\los$, we have  $\gamma(Sf)=\rho(f,fu,e)$ and so $\gamma=\rho^u$.
\end{proof}

In general, for an arbitrary regular semigroup $S$, two distinct principal cones 	$\rho^a$ and $\rho^b$ may be equal in $\los$ even when $a\neq b$. But when $S$ is Clifford, we proceed to show that it is not the case.

\begin{prop}\label{id:cone}
Let $S$ be a Clifford semigroup. Given two principal cones $\rho^a$ and $\rho^b$, we have $\rho^a=\rho^b$ if and only if $a=b$.
\end{prop}
\begin{proof}
Clearly when $a=b$, then $\rho^a=\rho^b$. Conversely suppose $\rho^a=\rho^b$. Then their vertices coincide and so, we have $Sa=Sb$, then since $S$ is Clifford and Green's relations coincide, we have $a\gh b$. Now let $e$ be the idempotent in $H_a=H_b$, the Green's $\gh$ class containing $a$ and $b$. Then $\rho^a(Se)=\rho(e,ea,e)=\rho(e,a,e)$. Similarly we get $\rho^b(Se)=\rho(e,b,e)$. Since the cones are equal, the corresponding morphism components at each vertex coincide. Hence using Proposition \ref{proeq}, we have $a=b$. 
\end{proof}

Recall from \cite[Section III.1]{nambooripad1994theory} that the set of all normal cones in a normal category forms a regular semigroup, under a natural binary operation. So, in particular, given the normal category $\los$, the set $\tls$ of all normal cones in the category $\los$ is a regular semigroup. Now, we proceed to characterise this semigroup when $S$ is Clifford.

\begin{thm}\label{los-to-s}
Let $S$ be a Clifford semigroup. Then the semigroup	$\tls$ of all normal cones in $\los$ is isomorphic to the semigroup $S$.
\end{thm}
\begin{proof}
Recall from \cite[Section III.3.2]{nambooripad1994theory} that the map $\bar{\rho}\colon a\mapsto \rho^a$ from a regular semigroup $S$ to the semigroup $\tls$ is a homomorphism. Now, when $S$ is Clifford, by \ref{pric:cone}, we have seen that every normal cone in $\tls$ is principal and hence the map $\bar{\rho}$ is surjective. Also, by  Proposition \ref{id:cone}, we see that the map $\bar{\rho}$ is surjective.   Hence the map $\bar{\rho}$ is an isomorphism from the Clifford semigroup $S$ to the semigroup $\tls$.  
\end{proof}

The above theorem characterises the semigroup of $\tls$ of all normal cones in $\los$; this naturally leads us to the complete description of the cross-connection structure of the semigroup $S$ as follows.

Recall from \cite[Section III.4]{nambooripad1994theory} that given a normal category, it has an associated dual category whose objects are certain set-valued functors and morphisms are natural transformations. Now we proceed to characterise the normal dual $\nls$ of the normal category $\los$ of principal left ideals of a Clifford semigroup $S$.

\begin{thm}\label{thmnls}
Let $S$ be a Clifford semigroup. Then the normal dual $\nls$ of the normal category $\los$ of principal left ideals in $S$ is isomorphic to the normal category $\mathbb{R}(S)$ of principal right ideals in $S$.
\end{thm}
\begin{proof}
It is known that \cite[Theorem III.25]{nambooripad1994theory} the normal dual $\nls$ of the normal category $\los$ is isomorphic to  the normal category $\mathbb{R}(\tls)$ of principal right ideals of the regular semigroup $\tls$.  When $S$ is a Clifford semigroup, by Theorem \ref{los-to-s}, we see that the semigroup $\tls$  is isomorphic to $S$ and hence $\nls$ is isomorphic to $\mathbb{R}(S)$ as normal categories.  
\end{proof}
Dually, we can easily prove the following results:
\begin{thm}\label{thmnrs}
Let $S$ be a Clifford semigroup. The semigroup $\trs$ of all normal cones in the category $\ros$ of all principal right ideals in $S$ is anti-isomorphic to the semigroup $S$. The normal dual $\nrs$ of the normal category $\ros$ is isomorphic to the normal category $\los$.
\end{thm}

Recall from \cite[Theorem IV.1]{nambooripad1994theory} that the cross-connection of a regular semigroup $S$ is defined as a quadruplet $(\los,\ros,\Gamma,\Delta)$ such that 
$\Gamma\colon\ros \to \nls$ and $\Delta\colon\los \to \nrs$ are functors satisfying certain properties. 
Now, using Theorems \ref{thmnls} and \ref{thmnrs}, we see that both the functors $\Gamma$ and $\Delta$ are in fact isomorphisms. And hence the cross-connection structure degenerates to isomorphisms of the associated normal categories in a Clifford semigroup $S$.

\section{Cross-connections of a semilattice}

Clearly, a semilattice is a Clifford semigroup. So, we specialise our discussion to a semilattice using the results in the previous section. The following theorem follows from Theorem \ref{los-to-s}.

\begin{thm}
Let $S$ be a semilattice. Then the semigroup	$\tls$ of all normal cones in $\los$ is isomorphic to the semilattice $S$.
\end{thm}

Theorems \ref{thmnls} and \ref{thmnrs} when applied to a semilattice can be unified as follows:
\begin{thm}\label{thmsemi}
Let $S$ be a semilattice. Then the normal dual $\nls$ of the normal category $\los$ of principal left ideals in $S$ is isomorphic to the normal category $\mathbb{R}(S)$ of principal right ideals in $S$. The normal dual $\nrs$ of the normal category $\ros$ is isomorphic to the normal category $\los$. 
\end{thm}

Hence, we see that both the cross-connections functors $\Gamma$ and $\Delta$ are isomorphisms, when we have a semilattice also. So, the cross-connection structure degenerates to isomorphisms of the associated normal categories in a semilattice, too.

\bibliography{books}
\bibliographystyle{plain}

\end{document}